\documentclass[12pt]{article}

\usepackage{amsmath, amsthm, amssymb, mathrsfs}
\usepackage{graphicx, tikz}
\usepackage{hyperref} 

\newtheorem{theorem}{Theorem}[section]
\newtheorem{definition}{Problem}[section]

\newtheorem{remark}{Remark}[section]

\begin{document}

\title{Uniqueness results on phaseless inverse scattering with a reference ball}

\author{
Deyue Zhang\thanks{School of Mathematics, Jilin University, Changchun,
P. R. China. {\it dyzhang@jlu.edu.cn}}\ \ and 
Yukun Guo\thanks{Department of Mathematics, Harbin Institute of Technology, Harbin, P. R. China. {\it ykguo@hit.edu.cn} (Corresponding author)}
}

\maketitle

\begin{abstract}
This paper is devoted to the uniqueness in inverse acoustic scattering problems for the Helmholtz equation with phaseless far-field data. Some novel techniques are developed to overcome the difficulty of translation invariance induced by a single incident plane wave.  In this paper, based on adding a reference ball as an extra artificial impenetrable obstacle (resp. penetrable homogeneous medium) to the inverse obstacle (resp. medium) scattering system and then using superpositions of a fixed plane wave and some point sources as the incident waves, we rigorously prove that the location and shape of the obstacle as well as its boundary condition or the refractive index can be uniquely determined by the modulus of far-field patterns. The reference ball technique in conjunction with the superposition of incident waves brings in several salient benefits. First, the framework of our theoretical analysis can be applied to both the inverse obstacle and medium scattering problems. Second, for inverse obstacle scattering, the underlying boundary condition could be of a general type and be uniquely determined. Finally, only a single frequency is needed. 
\end{abstract}

\noindent{\it Keywords}: inverse scattering, phaseless, uniqueness, Helmholtz equation, far field, reference ball


\section{Introduction}

The inverse scattering problems with the knowledge of the far-field pattern is of great importance in many areas of sciences and technology such as nondestructive testing, radar and sonar, and medical imaging (see, e.g., \cite{Colton}). However, in a variety of realistic applications, the measurements of the full data (both the intensity and phase) cannot be obtained, but instead only intensity information of the far-field pattern might be available. This leads to the study of inverse scattering problems with phaseless far-field data. This paper will be focused on the uniqueness in phaseless inverse acoustic scattering problems.

The main difficulty in showing the uniqueness roots in the translation invariance, i. e., the modulus of the far field pattern is invariant under translations \cite{Kress, Liu}. Moreover, as pointed out in \cite{Kress}, this difficulty cannot be overcome by using finitely many incident waves with different wave numbers or different incident directions. Under the condition that the scatterer is a small sound-soft ball, Liu and Zhang \cite{LZ09} proved that it can be uniquely determined by the modulus of the far-field datum measured at a fixed observation corresponding to a single incident plane wave. For the shape reconstruction of impenetrable obstacles, a number of numerical methods have been proposed in the literature, see, e.g., the Newton method \cite{Kress}, the nonlinear integral equation method \cite{Ivanyshyn1, Ivanyshyn2, Ivanyshyn3}, the
fundamental solution method \cite{KarageorghisAPNUM}, the hybrid method \cite{Lee2016} and the reverse time migration method \cite{CH17}. We also refer to
\cite{BLL2013, Bao2016, Li1, LLW17, KR16, Shin} for the reconstruction of the shape of a polyhedral obstacle, a convex sound-soft scatterer, a periodic grating profile and multi-scale sound-soft rough surfaces from phaseless far-field or near-field data.

There have been some investigations in the literature to break the translation invariance and determine the location of obstacles. In \cite{ZhangBo20171}, the superpositions of two plane waves with different incident directions were used as the incident fields and a recursive Newton-type iteration method in frequencies was proposed. Further, they developed this method to recover locally rough surfaces with phaseless far-field data in \cite{ZhangBo20172}. Recently, based on the method of superposition of two incident plane waves, they have proved in \cite{ZhangBo20173} that the obstacle and the refractive index of an inhomogeneous medium can be uniquely determined by the phaseless far-field patterns under the condition that the obstacle is sound-soft or non-absorbing impedance and the refractive index is real-valued with a positive lower bound or a negative upper bound. For inverse medium scattering problems, the translation invariance relation also holds, and we refer to \cite{Kli14, Kli17, KR17} for the relevant uniqueness results.

In this paper, some novel techniques are developed to get rid of the translation invariance in phaseless inverse scattering so that the uniqueness is obtainable. The rationale behind our methodology consists of two main ingredients: the incorporation of a reference ball and superposition of different types of incident waves. Motivated by the calibrating ball technique in \cite{Colton1, Colton2} and the reference ball technique in \cite{Li} for the linear sampling method and the method of superposition of two plane waves as incident waves  \cite{ZhangBo20171,  ZhangBo20173}, we first add a sound-soft ball as an extra artificial obstacle to the inverse obstacle scattering system. Then, the superposition of a fixed plane wave and some point sources is utilized as the incident wave, we prove that the obstacle with the boundary condition and the refractive index can be uniquely determined by the modulus of far-field patterns. The reference ball technique, with the aid of superposition of incident waves, brings in several benefits: it can deal with a general boundary condition; it is not required to know the boundary condition, which can be also uniquely determined as the obstacle by the phaseless far-field patterns. Finally, we would like to remark that the reference ball techniques in the literature are mainly used for the numerical purpose of improving the quality of reconstructions, and to our best knowledge, this is the first attempt to introduce this technique to the theoretical justifications of the uniqueness in phaseless inverse scattering.

The rest of this paper is organized as follows. In the next section, we present an introduction to the model problem and the reference ball technique. Section \ref{sec:obstacle} is devoted to the uniqueness results on phaseless inverse obstacle scattering. In Section \ref{sec:medium}, we present the uniqueness on phaseless inverse medium scattering. Some concluding remarks are given in Section \ref{sec:conclusion}.

\section{Problem setting}\label{sec:problem_setup}

We begin this section with the precise formulations of the model acoustic scattering problems. Let $D \subset\mathbb{R}^3$ be an open and simply connected domain with $C^2$ boundary $\partial D$. Denote by $\nu$ be the unit outward normal to $\partial D$ and by $\mathbb{S}^2:=\{x\in\mathbb{R}^3: |x|=1\}$ the unit sphere in $\mathbb{R}^3$. Consider the scattering of a given incoming wave $u^i(x,d)=\mathrm{e}^{\mathrm{i} k x\cdot d}$ by the scatterer $D$, where $d\in\mathbb{S}^2$ and $k>0$ are the incident direction and wave number, respectively. The obstacle scattering problem is to find the total field $u=u^i+u^s$ that satisfies the exterior boundary value problem (see \cite{Colton}):
\begin{eqnarray}
\Delta u+ k^2 u= 0\quad \mathrm{in}\ \mathbb{R}^3\backslash\overline{D},\label{eq:Helmholtz} \\
\mathscr{B}u= 0 \quad \mathrm{on}\ \partial D, \label{eq:boundary_condition} \\
\lim\limits_{r=|x|\rightarrow\infty} r\bigg(\displaystyle\frac{\partial u^s}{\partial r} -\mathrm{i} ku^s\bigg)=0, \label{eq:Sommerfeld}
\end{eqnarray}
where $u^s$ is the scattered field and the Sommerfeld radiation condition \eqref{eq:Sommerfeld} holds uniformly with respect to all directions $\hat{x}=x/|x|\in\mathbb{S}^2$. The boundary operator $\mathscr{B}$ in \eqref{eq:boundary_condition} signifies the following mixed boundary condition
\begin{eqnarray}\label{BC}
\mathscr{B}u=
\left\{
\begin{array}{lr}
u & \mathrm{on}\ \Gamma_D, \\
\displaystyle\frac{\partial u}{\partial \nu}+ \lambda u, & \mathrm{on}\ \Gamma_I,
\end{array}
\right.
\end{eqnarray}
where $\Gamma_D\cup\Gamma_I=\partial D, \Gamma_D\cap\Gamma_I=\emptyset, \lambda\in C(\Gamma_I)$ and $\mathrm{Im}\lambda \geq 0$. The mixed boundary condition \eqref{BC} is rather general in the sense that it covers the usual Dirichlet/sound-soft boundary condition ($\Gamma_I=\emptyset$), the Neumann/sound-hard boundary condition ($\Gamma_D=\emptyset$ and $\lambda=0$), and the impedance boundary condition ($\Gamma_D=\emptyset$ and $\lambda\neq 0$).

The medium scattering problem is to find the total field $u=u^i+u^s$ that satisfies
\begin{eqnarray}
\Delta u+ k^2n_D u=  0 \quad \text{in}\ \mathbb{R}^3, \label{eq:Helmholtz_D}\\
\lim\limits_{r=|x|\rightarrow\infty} r\bigg(\displaystyle\frac{\partial u^s}{\partial r}-\mathrm{i} ku^s\bigg)=0, \label{eq:Sommerfeld2}
\end{eqnarray}
where the refractive index $n_D(x)$ of the inhomogeneous medium is piecewise continuous with $\mathrm{Re}(n_D)>0$, $\mathrm{Im}(n_D)\geq0$ and $n_D(x)\equiv1$ for $x\notin D$.

\begin{figure}
	\centering
	\begin{tikzpicture}
	\pgfmathsetseed{8}
	\draw plot [smooth cycle, samples=6, domain={1:8}] (\x*360/8+5*rnd:0.5cm+1cm*rnd) node at (0,0) {$D$};
	\draw [->] (-1.6,1.4)--(-1,1) node at (-1,1.5) {$u^i$};
	\draw [->] (1,1)--(1.5, 1.5) node at (1,1.5) {$u^s$};
	\draw [dashed] (0, 0) circle (3cm) node at (3.2, 2.5) {$u^\infty(|x|\to \infty)$};
	\end{tikzpicture}
	\caption{An illustration of the model scattering problem.} \label{fig:model_problem}
\end{figure}
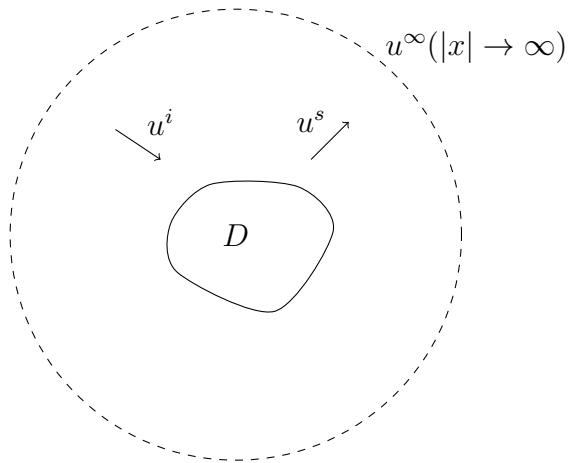

The direct scattering problems \eqref{eq:Helmholtz}--\eqref{eq:Sommerfeld} and \eqref{eq:Helmholtz_D}--\eqref{eq:Sommerfeld2} have a unique solution
$u\in H^1_{loc}(\mathbb{R}^3\backslash\overline{D})$ (see \cite{Cakoni, McLean}) and $u\in H^2_{loc}(\mathbb{R}^3)$ (see \cite{Colton}), respectively, and the scattered wave $u^s$ admits the
asymptotic behavior of the form
\begin{equation*}
u^s(x)=\frac{\mathrm{e}^{\mathrm{i} k|x|}}{|x|}\left\{ u^{\infty}(\hat{x})+\mathcal{O}\left(\frac{1}{|x|}\right) \right\}, \quad |x|\rightarrow\infty
\end{equation*}
uniformly in all directions $\hat{x}=x/|x|\in\mathbb{S}^2$. The analytic function $u^{\infty}(\hat{x})$ defined on the unit sphere $\mathbb{S}^2$ is called the far field pattern or scattering amplitude (see \cite{Colton}). For an illustration of the model scattering problem, we refer to Figure \ref{fig:model_problem}. In what follows, we shall also employ $u^{\infty}_D(\hat{x}, d; k)$ to indicate the dependence of the far field pattern $u^\infty(\hat{x})$ on the observation direction $\hat{x}$, the incident direction $d$, the obstacle or medium $D$, and the wave number $k$. Then, the phaseless inverse scattering problem is stated as follows.
\begin{definition}
	Let $D$ be either an impenetrable obstacle or a penetrable inhomogeneous medium. Given phaseless far field data
	$|u^{\infty}_D(\hat{x}, d; k)|$ for $\hat{x}, d\in \mathbb{S}^2$ and a fixed $k>0$, determine the location and shape $\partial D$ as well as the boundary condition $\mathscr{B}$ for the obstacle or the refractive index $n_D$ for the medium inclusion.
\end{definition}

There is a main difficulty in solving this problem: non-uniqueness, i.e., the location of the scatterer cannot be uniquely determined from the intensity-only far-field data. More specifically, for the shifted domain $D^h := \{x+h:  x \in D\}$ with a fixed vector $h\in \mathbb{R}^3$, the far field pattern $u^{\infty}_{D^h}(\hat{x}, d)$ satisfies the relation
\begin{equation}\label{eq:translation}
u^{\infty}_{D^{h}}(\hat{x}, d)=\mathrm{e}^{\mathrm{i} kh\cdot(d-\hat{x})}u^{\infty}_D(\hat{x}, d), \quad \hat{x}\in \mathbb{S}^2.
\end{equation}

As we mentioned, this ambiguity cannot be remedied by using finitely many incident waves with different wave numbers or different incident directions. In the following, we aim to tackle the translation invariance by adding a reference ball and utilizing the superposition of incident waves. To this end, we first introduce a reference ball $B$ as an extra artificial object to the scattering system such that $\overline{D}\cap\overline{B}=\emptyset$. For the obstacle scattering problem, the reference ball is sound-soft, while for the medium scattering problem, the reference ball is a homogeneous medium
$B=B(x_0,R):=\{x\in \mathbb{R}^3: |x-x_0|<R\}$ with refractive index $n_0^2$ such that
\begin{equation}\label{assume}
R<\frac{\pi}{2k(n_0+1)},
\end{equation}
where $n_0\neq1$ is a positive constant. The refractive index of the medium scattering system is defined by
\begin{equation*}
n(x):=\left\{
\begin{array}{ll}
n_D(x),  & x\in D, \\
n_0^2,  & x\in B, \\
1,      & x\in \mathbb{R}^3\backslash (\overline{D}\cup \overline{B}).
\end{array}
\right.
\end{equation*}

Assume that $P$ is a simply-connected convex polyhedron such that $P \subset\mathbb{R}^3\backslash (\overline{D}\cup \overline{B})$ and $k^2$ is not a Dirichlet eigenvalue of $-\Delta$ in $P$. We remark that this configuration of $P$ can be achieved, for example, in the case that $P$ is contained in a ball whose radius is less than $\pi/k$.  The boundary of $P$ is composed of finitely many 2D polygons $\Pi_\ell, \ell=1,\cdots, N$, namely, $\Pi:=\partial P=\cup_{\ell=1}^N \Pi_\ell$.
Then let us consider the superposition of a plane wave $u^i$ and a point source $v^i$ as the incident wave:
\begin{equation}\label{incident}
u^i(x, d)+v^i(x, z)=\mathrm{e}^{\mathrm{i} k x\cdot d}+\Phi(x, z),
\end{equation}
where $z \in \Pi$ denotes the point source location, and
\begin{equation*}
\Phi (x,z)=\frac{\mathrm{e}^{\mathrm{i} k|x-z|}}{4\pi |x-z|}
\end{equation*}
is the fundamental solution to the Helmholtz equation.  

Let the pairs $\{u^s_{D\cup B}(x,d), u^{\infty}_{D\cup B}(\hat{x},d)\}$ and $\{v^s_{D\cup B}(x,z), v^{\infty}_{D\cup B}(\hat{x},z)\}$ be the scattered field and its far-field pattern generated by $D\cup B$ corresponding to the incident field $u^i$ and $v^i$, respectively. Then, by the linearity of direct scattering problem, the scattered field and the far-field pattern generated by $D\cup B$ and the incident wave $u^i+v^i$ defined in \eqref{incident} are given by $u^s_{D\cup B}(x,d)+v^s_{D\cup B}(x,z), x\in \mathbb{R}^3\backslash (\overline{D}\cup \overline{B})$, and $u^{\infty}_{D\cup B}(\hat{x},d)+v^{\infty}_{D\cup B}(\hat{x},z)$, $ \hat{x}\in \mathbb{S}^2$, respectively.

Under the above configurations, now we formulate the phaseless inverse scattering problems as follows.

\begin{definition}[Phaseless inverse obstacle scattering with a reference ball]\label{prob:obstacle}
	Let $D$ and $B$ be, respectively, the impenetrable obstacle with boundary condition $\mathscr{B}$ and the sound-soft reference ball. Given the phaseless far field data   
	\begin{equation*}
	\begin{array}{ll}
	& \{|u^{\infty}_{D\cup B}(\hat{x}, d_0)|: \hat{x}\in \mathbb{S}^2\}, \\
	& \{|v^{\infty}_{D\cup B}(\hat{x}, z)|:  \hat{x}\in \mathbb{S}^2, z\in \Pi\}, \\
	& \{|u^{\infty}_{D\cup B}(\hat{x}, d_0)+v^{\infty}_{D\cup B}(\hat{x}, z)|: \hat{x}\in \mathbb{S}^2, z\in \Pi\}.
	\end{array}
	\end{equation*} 
	for a fixed wavenumber $k>0$ and a fixed incident direction $d_0\in \mathbb{S}^2$, determine the location and shape $\partial D$ as well as the boundary condition $\mathscr{B}$ for the obstacle.
\end{definition}

\begin{definition}[Phaseless inverse medium scattering with a reference ball]\label{prob:medium}
	Let $D$ and $B$ be, respectively, the inhomogeneous medium with refractive index $n_D$ and the reference medium ball. Given the phaseless far field data   
	\begin{equation*}
	\begin{array}{ll}
	& \{|u^{\infty}_{D\cup B}(\hat{x}, d_0)|: \hat{x}\in \mathbb{S}^2\}, \\
	& \{|v^{\infty}_{D\cup B}(\hat{x}, z)|:  \hat{x}\in \mathbb{S}^2, z\in \Pi\}, \\
	& \{|u^{\infty}_{D\cup B}(\hat{x}, d_0)+v^{\infty}_{D\cup B}(\hat{x}, z)|: \hat{x}\in \mathbb{S}^2, z\in \Pi\}.
	\end{array}
	\end{equation*} 
	for a fixed wavenumber $k>0$ and a fixed incident direction $d_0\in \mathbb{S}^2$, determine $n_D$.
\end{definition}

We refer to Figure \ref{fig:illustration} for an illustration of the geometry setup of Problems \ref{prob:obstacle} and \ref{prob:medium}. The uniqueness of these problems will be analyzed in the subsequent sections. 

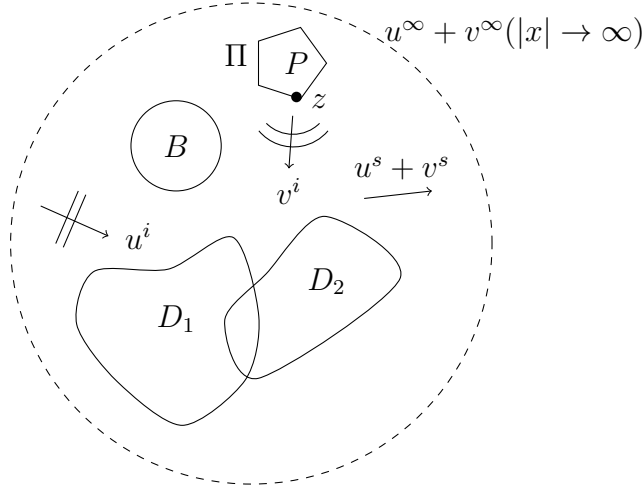
\begin{figure}
	\centering
	\newdimen\R 
	\R=0.5cm
	\begin{tikzpicture}
	\pgfmathsetseed{3}
	\draw plot [smooth cycle, samples=8, domain={1:8}] (\x*360/8+5*rnd:0.5cm+1cm*rnd) node at (0,0) {$D_1$};
	\pgfmathsetseed{9}
	\draw plot [smooth cycle, samples=5, domain={1:5}, xshift=2cm] (\x*360/8+5*rnd:0.5cm+1cm*rnd) node at (2,0.5) {$D_2$};
	\draw (0, 2.3) circle (0.6cm) node at (0,2.3) {$B$}; 
	\draw [->] (-1.8,1.5)--(-0.9,1.1) node at (-0.5,1.1) {$u^i$};
	\draw (-1.6, 1)--(-1.3, 1.7);
	\draw (-1.5, 0.95)--(-1.2, 1.65);
	\draw [->] (2.5,1.6)--(3.4, 1.7) node at (3,2) {$u^s+v^s$};
	\fill [black] (1.6,2.95) circle (2pt) node at (1.9,2.9) {$z$};  
	\draw (1.2, 2.6) arc(225:315:0.5cm);
	\draw (1.1, 2.5) arc(220:320:0.6cm);
	\draw [->] (1.55, 2.7)--(1.5, 2) node at (1.5, 1.7) {$v^i$};
	\draw [dashed] (1, 1) circle (3.2cm) node at (4.5, 3.8) {$u^\infty+v^\infty(|x|\to \infty)$};
	
	\draw[xshift=3\R, yshift=6.8\R] (0:\R) \foreach \x in {72,144,...,359} {-- (\x:\R)} -- cycle (90:\R); 
	\draw node at (1.6, 3.4) {$P$};
	\draw node at (0.8, 3.5) {$\Pi$};
	\end{tikzpicture}
	\caption{An illustration of the reference ball technique.} \label{fig:illustration}
\end{figure}

\section{Inverse obstacle scattering}\label{sec:obstacle}

We are now in a position to present the main result on uniqueness in phaseless inverse obstacle scattering. The following theorem shows that Problem \ref{prob:obstacle} admits a unique solution, i.e., the geometrical and physical information of the scatterer boundary can be simultaneously and uniquely determined from the modulus of far-field patterns.

\begin{theorem}\label{Thm1}
	Let $D_1$ and $D_2 $ be two obstacles with boundary conditions $\mathscr{B}_1$ and $\mathscr{B}_2$, respectively. Given a fixed $d_0 \in \mathbb{S}^2$, the scattered field and its far-field pattern with respect to the incident field $u^i(x, d_0)$ and $v^i(x, z)$ are denoted by $u^s_{D_j\cup B}(x,d_0)$, $u^\infty_{D_j\cup B}(\hat{x},d_0)$, $v^s_{D_j\cup B}(x,z)$ and $v^\infty_{D_j\cup B}(\hat{x},z)$, $j=1,2$, respectively. If the far-field patterns satisfy that
	\begin{align}
	|u^\infty_{D_1\cup B}(\hat{x},d_0)|= &|u^\infty_{D_2\cup B}(\hat{x},d_0)|, \quad \forall \hat{x}  \in \mathbb{S}^2, \label{obstacle_condition1} \\
	|v^\infty_{D_1\cup B}(\hat{x},z)|= &|v^\infty_{D_2\cup B}(\hat{x},z)|, \quad \forall (\hat{x}, z) \in \mathbb{S}^2\times \Pi \label{obstacle_condition2}
	\end{align}
	and
	\begin{equation}\label{obstacle_condition3}
	|u^{\infty}_{D_1\cup B}(\hat{x},d_0)+v^{\infty}_{D_1\cup B}(\hat{x},z)|=|u^{\infty}_{D_2\cup B}(\hat{x},d_0)+v^{\infty}_{D_2\cup B}(\hat{x},z)|
	\end{equation}
	for all $(\hat{x}, z) \in \mathbb{S}^2\times \Pi$, then $D_1=D_2$ and $\mathscr{B}_1=\mathscr{B}_2$.
\end{theorem}
\begin{proof}
	From  \eqref{obstacle_condition1}, \eqref{obstacle_condition2} and \eqref{obstacle_condition3}, we have for all $\hat{x}\in\mathbb{S}^2, z\in\Pi$
	\begin{equation}\label{Thm1equality1}
	\mathrm{Re}\left\{u^{\infty}_{D_1\cup B}(\hat{x},d_0) \overline{v^{\infty}_{D_1\cup B}(\hat{x},z)}\right\}
	=\mathrm{Re}\left\{u^{\infty}_{D_2\cup B}(\hat{x},d_0) \overline{v^{\infty}_{D_2\cup B}(\hat{x},z)}\right\}.
	\end{equation}
	where the overline denotes the complex conjugate. According to \eqref{obstacle_condition1} and \eqref{obstacle_condition2}, we denote
	\begin{equation*}
	u^{\infty}_{D_j\cup B}(\hat{x},d_0)=r(\hat{x},d_0) \mathrm{e}^{\mathrm{i} \alpha_j(\hat{x},d_0)},\quad
	v^{\infty}_{D_j\cup B}(\hat{x},z)=s(\hat{x},z) \mathrm{e}^{\mathrm{i} \beta_j(\hat{x},z)},\quad j=1,2,
	\end{equation*}
	where $r(\hat{x},d_0)=|u^{\infty}_{D_j\cup B}(\hat{x},d_0)|, s(\hat{x},z)=|v^{\infty}_{D_j\cup B}(\hat{x},z)|$,
	$\alpha_j(\hat{x},d_0)$ and $\beta_j(\hat{x},z)$, are real-valued functions,  $j=1,2$.
	
	By using $s(\hat{x},z)\not\equiv 0$ for $\hat{x} \in \mathbb{S}^2$ and $z \in \Pi_1$, and its continuity, we know that there exist open sets $S_1\subset  \mathbb{S}^2$ and $\Pi_1^0 \subset  \Pi_1$ such that $s(\hat{x},z)\neq 0$, $\forall (\hat{x},z)\in S_1\times \Pi_1^0$. From the analyticity of $v^{\infty}_{D_j\cup B}(\hat{x},z)$ with respect to $\hat{x} \in \mathbb{S}^2$, we see $s(\hat{x},z)\not\equiv 0$ for $\hat{x} \in S_1$ and $z \in \Pi_2$. Similarly, we have $s(\hat{x},z)\neq 0$, $\forall (\hat{x},z)\in S_2\times \Pi_2^0$, where $S_2\subset S_1$ and $\Pi_2^0 \subset  \Pi_2$. As an analogy, we obtain that $s(\hat{x},z)\neq 0$, $\forall (\hat{x},z)\in \tilde{S}\times \Pi^0$, where open sets $\tilde{S}\subset \mathbb{S}^2$, $\Pi^0=\cup_{\ell=1}^N\Pi_\ell^0$ and $\Pi_\ell^0 \subset  \Pi_\ell$($\ell=1,\cdots,N$). Since $u^{\infty}_{D_j\cup B}(\hat{x},d_0)$ is analytic with respect to $\hat{x} \in \mathbb{S}^2$,  $u^{\infty}_{D_j\cup B}(\hat{x},d_0)\not\equiv 0$ for $\hat{x} \in \mathbb{S}^2$, and $d_0$ is fixed,	we have $r(\hat{x},d_0)\not\equiv 0$ on $\tilde{S}_1$. Again, the continuity leads to $r(\hat{x},d_0)\neq 0$ on an open set $S \subset \tilde{S}$. Therefore, we have $r(\hat{x},d_0)\neq 0$, $s(\hat{x},z)\neq 0,\ \forall (\hat{x}, z) \in S\times \Pi^0$. This, together with \eqref{Thm1equality1}, implies  
	\begin{equation*}
	\cos[\alpha_1(\hat{x},d_0)-\beta_{1}(\hat{x},z)]=\cos[\alpha_2(\hat{x},d_0)-\beta_{2}(\hat{x},z)], \quad \forall (\hat{x}, z) \in S\times \Pi^0.
	\end{equation*}
	Hence,
	\begin{eqnarray}\label{Thm1equality2}
	\alpha_1(\hat{x},d_0)-\alpha_2(\hat{x},d_0)=
	\beta_{1}(\hat{x},z)-\beta_{2}(\hat{x},z)+ 2m\pi, \quad \forall (\hat{x}, z) \in S\times \Pi^0
	\end{eqnarray}
	or
	\begin{eqnarray}\label{Thm1equality3}
	\alpha_1(\hat{x},d_0)+\alpha_2(\hat{x},d_0)=
	\beta_{1}(\hat{x},z)+\beta_{2}(\hat{x},z)+ 2m\pi, \quad \forall (\hat{x}, z) \in S\times \Pi^0
	\end{eqnarray}
	with some $m \in \mathbb{Z}$.
	
	First, let us consider the case \eqref{Thm1equality2}. Since $d_0$ is fixed, we can define $\gamma(\hat{x}):=\alpha_1(\hat{x},d_0)-\alpha_2(\hat{x},d_0)- 2m\pi$
	for $ \hat{x} \in S$, and then, we have for all $ (\hat{x}, z) \in S\times \Pi^0$
	\begin{equation*}
	v^{\infty}_{D_1\cup B}(\hat{x},z)=s(\hat{x},z)\mathrm{e}^{\mathrm{i} \beta_{1}(\hat{x},z)}=s(\hat{x},z)\mathrm{e}^{\mathrm{i} \beta_{2}(\hat{x},z)+\mathrm{i} \gamma(\hat{x})}=\mathrm{e}^{\mathrm{i} \gamma(\hat{x})} v^{\infty}_{D_2\cup B}(\hat{x},z).
	\end{equation*}
	From the mixed reciprocity relation \cite[Theorem 3.16]{Colton}, we have
	\begin{equation*}
	4\pi v^{\infty}_{D_j\cup B}(\hat{x},z)= u^{s}_{D_j\cup B}(z,-\hat{x}),
	\end{equation*}
	and thus,
	\begin{equation*}
	u^{s}_{D_1\cup B}(z,-\hat{x})
	=\mathrm{e}^{\mathrm{i} \gamma(\hat{x})} u^{s}_{D_2\cup B}(z,-\hat{x}), \quad \forall (\hat{x}, z) \in S\times \Pi^0.
	\end{equation*}
	Further, for every $-d\in S$, it holds that $u^{s}_{D_1\cup B}(z,d)=\mathrm{e}^{\mathrm{i} \gamma(-d)} u^{s}_{D_2\cup B}(z,d)$ for $z\in \Pi_\ell^0$. By using the analyticity of
	$u^{s}_{D_j\cup B}(z,d)$($j=1,2$), we have $u^{s}_{D_1\cup B}(z,d)=\mathrm{e}^{\mathrm{i} \gamma(-d)} u^{s}_{D_2\cup B}(z,d)$ for $z\in \Pi_\ell$.
	Let $w(x,d)=u^{s}_{D_1\cup B}(x,d)-\mathrm{e}^{\mathrm{i} \gamma(-d)}  u^{s}_{D_2\cup B}(x,d)$, then
	\begin{equation*}
	\left\{
	\begin{array}{lr}
	\Delta w+ k^2 w=0 &\mathrm{in}\ P, \\
	w=0  &\mathrm{on}\ \partial P.
	\end{array}
	\right.
	\end{equation*}
	By the assumption of $P$ that $k^2$ is not a Dirichlet eigenvalue of $-\Delta$ in $P$,  we deduce $w=0$ in $P$. Now, the analyticity of
	$u^{s}_{D_j\cup B}(x,d)$($j=1,2$) with respect to $x$ leads to
	\begin{equation*}
	u^{s}_{D_1\cup B}(x,d)=\mathrm{e}^{\mathrm{i} \gamma(-d)}  u^s_{D_2\cup B}(x,d),
	\quad \forall x \in \mathbb{R}^3\backslash (D_1\cup D_2\cup B).
	\end{equation*}
	By using the sound-soft boundary condition $u^{s}_{D_j\cup B}(x,d)=-\mathrm{e}^{\mathrm{i} k x\cdot d}, x\in\partial B$ ($j=1,2$), we have
	\begin{equation*}
	-\mathrm{e}^{\mathrm{i} k x\cdot d}=-\mathrm{e}^{\mathrm{i} \gamma(-d)}  \mathrm{e}^{\mathrm{i} k x\cdot d}, \quad x \in \partial B.
	\end{equation*}
	Hence, $\mathrm{e}^{\mathrm{i} \gamma(-d)} \equiv1$, and
	\begin{equation*}
	u^{\infty}_{D_1\cup B}(\hat{x},d)= u^{\infty}_{D_2\cup B}(\hat{x},d),
	\quad \forall (\hat{x},-\tilde{d})\in \mathbb{S}^2\times S.
	\end{equation*}
	Further, the reciprocity relation and the analyticity of $u^{\infty}_{D_j\cup B}(-d,-\hat{x})$($j=1,2$) with respect to $d \in \mathbb{S}^2$ implies that the far field patterns coincide, i.e.,
	\begin{eqnarray}\label{coincide}
	u^{\infty}_{D_1\cup B}(\hat{x},d)= u^{\infty}_{D_2\cup B}(\hat{x},d),
	\quad \forall \hat{x}, d \in \mathbb{S}^2.
	\end{eqnarray}
	
	Next we are going to show that the case \eqref{Thm1equality3} does not hold. By using a similar argument, we have $u^{s}_{D_1\cup B}(x,d)=\mathrm{e}^{\mathrm{i} \eta(-d)} \overline{u^{s}_{D_2\cup B}(x,d)}$ for  $-d  \in S$ and $x \in \mathbb{R}^3\backslash (\overline{D_1}\cup\overline{D_2}\cup \overline{B})$. From the analyticity of $u^{\infty}_{D_1\cup B}(\hat{x},d)$  with respect to $\hat{x} \in \mathbb{S}^2$, $u^{\infty}_{D_1\cup B}(\hat{x},d)\not\equiv 0$ and its continuity, it follows that there exist open sets $\widetilde{U}_1\subset \mathbb{S}^2$, $\widetilde{U}_2\subset S$ such that $u^{\infty}_{D_1\cup B}(\hat{x},d)\neq 0$, $\forall (\hat{x},-d)\in \widetilde{U}_1\times \widetilde{U}_2$. Further, we consider $u^{\infty}_{D_2\cup B}(\hat{x},d)$ for $ (\hat{x},-d)\in \widetilde{U}_1\times \widetilde{U}_2$, and by the same way, it can be deduced that there exist open sets $U_1\subset \widetilde{U}_1$, $U_2\subset \widetilde{U}_2$ such that $u^{\infty}_{D_1\cup B}(\hat{x},d)\neq 0$ and $u^{\infty}_{D_2\cup B}(\hat{x},d)\neq 0$, $\forall (\hat{x},-d)\in U_1\times U_2$.
	By taking $\tilde{x}\in U_1$, $-\tilde{d}\in U_2$, $x=\rho \tilde{x}$, and using the definition of the far field pattern (see \cite[Theorem 2.6]{Colton}), we obtain
	\begin{equation*}
	\lim\limits_{\rho\rightarrow \infty} \rho\mathrm{e}^{-\mathrm{i} k\rho} u^{s}_{D_1\cup B}(\rho\tilde{x},\tilde{d})=u^{\infty}_{D_1\cup B}(\tilde{x},\tilde{d}) 
	\end{equation*}
	and
	\begin{equation*}
	\lim\limits_{\rho\rightarrow \infty} \rho\mathrm{e}^{\mathrm{i} k\rho} \overline{u^{s}_{D_2\cup B}(\rho\tilde{x},\tilde{d})}=\overline{u^{\infty}_{D_2\cup B}(\tilde{x},\tilde{d})}.
	\end{equation*}
	Further, noticing $u^{s}_{D_1\cup B}(\rho\tilde{x},\tilde{d})=\mathrm{e}^{\mathrm{i} \eta(-\tilde{d})} \overline{u^{s}_{D_2\cup B}(\rho\tilde{x},\tilde{d})}$ and $u^{\infty}_{D_j\cup B}(\tilde{x},\tilde{d})\neq 0$ ($j=1,2$),
	we have 
	\begin{equation*}
	\lim\limits_{\rho\rightarrow \infty}  \mathrm{e}^{2\mathrm{i} k\rho}  = \frac{\mathrm{e}^{\mathrm{i} \eta(-\tilde{d})} \overline{u^{\infty}_{D_2\cup B}(\tilde{x},\tilde{d})}}{u^{\infty}_{D_1\cup B}(\tilde{x},\tilde{d})},
	\end{equation*}
	which is a contradiction. Hence, the case \eqref{Thm1equality3} does not hold.
	
	Now, we go on considering \eqref{coincide}. By Theorem 5.6 in \cite{Colton}, we have $D_1=D_2=D$. Moreover, we claim that $\mathscr{B}_1=\mathscr{B}_2$. Otherwise, if $D=D_1=D_2$ and $\mathscr{B}_1\neq\mathscr{B}_2$, then from \eqref{BC}, we have that on a part of the boundary $\partial D$, denoted by $\Gamma$, the total field $u=u_1=u_2$ satisfies two different boundary conditions, that is,
	\begin{equation*}
	u=\frac{\partial u}{\partial \nu}+ \lambda u=0 \quad \mathrm{on}\ \Gamma
	\end{equation*}
	or
	\begin{equation*}
	\frac{\partial u}{\partial \nu}+ \lambda_1 u=\frac{\partial u}{\partial \nu}+ \lambda_2 u=0 \quad \mathrm{on}\ \Gamma
	\end{equation*}
	with $\lambda_1\neq \lambda_2$. For the two cases, it is readily to see that $u=\partial u/\partial \nu=0$ on $\Gamma$. By Holmgren's theorem \cite[Theorem 2.3]{Colton} and the boundary condition we obtain that $u=0$ in $\mathbb{R}^3\backslash (D\cup B)$. This leads to the contradiction that the incident field must satisfy the radiation condition. Hence, $\mathscr{B}_1=\mathscr{B}_2$.
\end{proof}

\begin{remark}
	The impossibility of \eqref{Thm1equality3} in the above proof could be justified alternatively. For example, by using a similar argument, we have
	$u^\infty_{D_1\cup B}(\hat{x},d)=\mathrm{e}^{\mathrm{i} \eta} \overline{u^\infty_{D_2\cup B}(\hat{x},d)}$
	for all $\hat{x}, d  \in \mathbb{S}^2$ with a real constant $\eta$. And thus
	$u^s_{D_1\cup B}(x,d)=\mathrm{e}^{\mathrm{i} \eta}\overline{u^s_{D_2\cup B}(x,d)}$ for all
	$x \in \mathbb{R}^3\backslash (D_1\cup D_2\cup B),\ d\in \mathbb{S}^2$.
	Substituting this into the sound-soft boundary condition on $\partial B$, we see $
	-\mathrm{e}^{\mathrm{i} k x\cdot d}=-\mathrm{e}^{\mathrm{i} \eta} \mathrm{e}^{-\mathrm{i} k x\cdot d}, \ x \in \partial B$, which is a contradiction. Therefore, the case \eqref{Thm1equality3} does not hold.
\end{remark}

\begin{remark}
	For inverse obstacle scattering problem in two dimensions, justifications of the uniqueness in Theorem \ref{Thm1} can be carried over in a similar way. In other words, by appropriate modifications of the fundamental solution and the radiation condition, an analogous assertion of uniqueness can be established in two dimensions.
\end{remark}

\begin{remark}
	We would like to point out that the a similar result on uniqueness can also be obtained by using the superposition of a fixed point source and some point sources as the incident fields on the scattering system with the reference ball.
\end{remark}

\section{Inverse medium scattering}\label{sec:medium}

Now we present the uniqueness results concerning phaseless inverse medium scattering.

\begin{theorem}\label{Thm2}
	Let $n_{D_1}$ and $n_{D_2}$ be the refractive index of medium $D_1$ and $D_2$, respectively. Given a fixed $d_0 \in \mathbb{S}^2$, the scattered field and its far-field pattern with respect to the incident field $u^i(x, d_0)$ and $v^i(x, z)$ are denoted by $u^s_{D_j\cup B}(x,d_0)$, $u^\infty_{D_j\cup B}(\hat{x},d_0)$,
	$v^s_{D_j\cup B}(x,z)$ and $v^\infty_{D_j\cup B}(\hat{x},z)$, $j=1,2$, respectively. If the far-field patterns satisfy that
	\begin{align}
	|u^{\infty}_{D_1\cup B}(\hat{x},d_0)|= &|u^{\infty}_{D_2\cup B}(\hat{x},d_0)|, \quad \forall \hat{x}  \in \mathbb{S}^2, \label{condition1} \\
	|v^{\infty}_{D_1\cup B}(\hat{x},z)|=&|v^{\infty}_{D_2\cup B}(\hat{x},z)|, \quad \forall (\hat{x},z) \in \mathbb{S}^2\times \Pi \label{condition2}
	\end{align}
	and
	\begin{eqnarray}\label{condition3}
	|u^{\infty}_{D_1\cup B}(\hat{x},d_0)+v^{\infty}_{D_1\cup B}(\hat{x},z)|=|u^{\infty}_{D_2\cup B}(\hat{x},d_0)+v^{\infty}_{D_2\cup B}(\hat{x},z)|
	\end{eqnarray}
	for all $(\hat{x}, z) \in \mathbb{S}^2\times \Pi$, then $n_{D_1}=n_{D_2}$.
\end{theorem}
\begin{proof}
	Using the mixed reciprocity relation for inhomogeneous medium \cite[Theorem 2.2.4]{Pot01} and following the same arguments in the proof of Theorem \ref{Thm1}, we have
	\begin{equation}\label{case1}
	u^s_{D_1\cup B}(x,d)=\mathrm{e}^{\mathrm{i} \gamma(-d)} u^s_{D_2\cup B}(x,d),
	\quad \forall x \in \mathbb{R}^3\backslash (D_1\cup D_2\cup B),\ -d  \in S,
	\end{equation}
	and for $x\in \partial B, -d  \in   S$,
	\begin{equation*}
	u^{s}_{D_1\cup B}(x,d)=\mathrm{e}^{\mathrm{i} \gamma(-d)}  u^{s}_{D_2\cup B}(x,d),\quad
	\frac{\partial u^s_{D_1\cup B}(x,d)}{\partial \nu}=\mathrm{e}^{\mathrm{i} \gamma(-d)}\frac{\partial u^s_{D_2\cup B}(x,d)}{\partial\nu}.
	\end{equation*}
	From the fact that the total field $u_{D_j\cup B}(x,d)=u^i(x,d)+u^s_{D_j\cup B}(x,d)$ for $j=1,2$ satisfies
	$\Delta u_{D_j\cup B}+ k^2n_0^2 u_{D_j\cup B}=0$ in $B$, we obtain that $\psi:=u_{D_1\cup B}(x,d)-\mathrm{e}^{\mathrm{i} \gamma(-d)} u_{D_2\cup B}(x,d)$ satisfies
	\begin{equation*}
	\Delta \psi+ k^2n_0^2 \psi=0 \quad \mathrm{in} \ B.
	\end{equation*}
	Now, for any $\phi$ satisfying the Helmholtz equation $\Delta \phi+ k^2n_0^2 \phi=0 $ in $B$,
	by using Green's second theorem for $\psi$ and $\phi$, we have
	\begin{equation*}
	\int_{\partial B}\left(\phi\frac{\partial\psi}{\partial\nu}-\psi\frac{\partial\phi}{\partial\nu}\right) \mathrm{d} s=\int_{B}(\phi\Delta \psi- \psi\Delta \phi)\,\mathrm{d} x=0.
	\end{equation*}
	Further, from $\psi=(1-\mathrm{e}^{\mathrm{i} \gamma(-d)} )u^i$ and $\frac{\partial\psi}{\partial\nu}=(1-\mathrm{e}^{\mathrm{i} \gamma(-d)} )\frac{\partial u^i}{\partial\nu}$ on $\partial B$, it follows that
	\begin{equation*}
	(1-\mathrm{e}^{\mathrm{i} \gamma(-d)} )\int_{\partial B}\left(\phi\frac{\partial u^i}{\partial\nu}-u^i\frac{\partial\phi}{\partial\nu}\right)\mathrm{d} s=0.
	\end{equation*}
	Again, Green's second theorem yields
	\begin{eqnarray}
	0=(1-\mathrm{e}^{\mathrm{i} \gamma(-d)} )\int_{\partial B}\left(\phi\frac{\partial u^i}{\partial\nu}-u^i\frac{\partial\phi}{\partial\nu}\right)\mathrm{d} s \nonumber\\
	\ \ = (1-\mathrm{e}^{\mathrm{i} \gamma(-d)} )\int_{B}(\phi\Delta u^i- u^i\Delta \phi)\,\mathrm{d} x \nonumber\\
	\ \ = (1-\mathrm{e}^{\mathrm{i} \gamma(-d)} ) k^2(n_0^2-1)\int_{B}\phi u^i\,\mathrm{d} x \nonumber\\
	\ \ = (1-\mathrm{e}^{\mathrm{i} \gamma(-d)} ) k^2(n_0^2-1)\int_B \phi(x)\mathrm{e}^{\mathrm{i} k x\cdot d}\,\mathrm{d} x, \quad \forall \ -d \in S.\nonumber
	\end{eqnarray}
	
	In the following, without loss of generality, let $x_0=(0,0,0)$ be the origin. By using $n_0\neq1$ and taking $\phi(x)=\mathrm{e}^{\mathrm{i} kn_0 x\cdot d }$,  we have
	\begin{equation}\label{relation1}
	(1-\mathrm{e}^{\mathrm{i} \gamma(-d)} ) \int_B \mathrm{e}^{\mathrm{i} k(n_0+1)  x\cdot d}\,\mathrm{d} x =0.
	\end{equation}
	It follows from \eqref{assume} that $|k(n_0+1) x\cdot d|\leq k(n_0+1) |x|< \pi/2$ for $|x|<R$, and thus, $\cos(k(n_0+1)x\cdot d)>0$ for $|x|<R$,
	which means
	\begin{equation*}
	\mathrm{Re}\int_B  \mathrm{e}^{\mathrm{i} k(n_0+1) x\cdot d}\,\mathrm{d} x
	=\int_B  \cos(k(n_0+1)x\cdot d)\,\mathrm{d} x>0.
	\end{equation*}
	This, together with \eqref{relation1}, leads to $\mathrm{e}^{\mathrm{i} \gamma(-d)} =1$. Hence, by using again the reciprocity relation and the analyticity of $u^{\infty}_{D_j\cup B}(-d,-\hat{x})$($j=1,2$) with respect to $d \in \mathbb{S}^2$, it follows that the far field patterns coincide, and in terms of \cite[Theorem 10.5]{Colton}, we obtain
	$n_{D_1}=n_{D_2}$. The proof is complete.
\end{proof}

\section{Conclusion}\label{sec:conclusion}

In this paper, we established some uniqueness results on inverse scattering problems for the Helmholtz equation with phaseless far-field data. The crux of our proof is the combination of the reference ball technique and the superposition of different incidence waves, which leads to the uniqueness on the general cases of obstacle scattering problems. We proved that the obstacle with its boundary condition can be uniquely determined by the modulus of far-field patterns. In fact, the reference ball and superposition of incident waves play the crucial role of calibrating the scattering system such that the translation invariance does not occur.    

Based on the ideas in this paper, our future work consist in the uniqueness results of the phaseless inverse elastic scattering problems, as well as the development of effective inversion algorithms with phaseless data.

\section*{Acknowledgements}

The first author was supported by NSFC grant 11671170 and the second author was supported by NSFC grants 11601107, 41474102 and 11671111.



\begin{thebibliography}{10}
	
	
	\bibitem{BLL2013} G. Bao, P. Li and J. Lv, Numerical solution of an inverse diffraction grating problem from phaseless data {\it J. Opt. Soc. Am. A} {\bf 30}, 293--299, 2013.
	
	\bibitem{Bao2016} G. Bao and L. Zhang, Shape reconstruction of the multi-scale rough surface from multi-frequency phaseless data {\it Inverse Problems} {\bf 32}, 085002, 2016.
	
	\bibitem{Cakoni} F. Cakoni, D. Colton and P. Monk, The direct and inverse scattering problem for partially coated obstacles  {\it Inverse Problems} {\bf 17}, 1997--2015, 2001.
	
	\bibitem{CH17} Z. Chen and G. Huang, Phaseless imaging by reverse time migration: acoustic waves {\it Numer. Math. Theor. Meth. Appl.}, {\bf 10}, 1--21, 2017.
	
	\bibitem{Colton1} D. Colton, J. Coyle and P. Monk, Recent developments in inverse acoustic scattering theory {\it SIAM Rev.} {\bf 42}, 369--414, 2000.
	
	\bibitem{Colton2} D. Colton, K. Giebermann and P. Monk, A regularized sampling method for solving three-dimensional inverse scattering problems {\it SIAM J. Sci. Comput.} {\bf 21}, 2316--2330, 2000.
	
	\bibitem{Colton} D. Colton and R. Kress {\it Inverse Acoustic and Electromagnetic Scattering Theory} {\it 3rd ed}. (New York: Springer-Verlag), 2013.
	
	\bibitem{Ivanyshyn1} O. Ivanyshyn, Shape reconstruction of acoustic obstacles from the modulus of the far field pattern  {\it Inverse Probl. Imaging}  {\bf 1},  609--622, 2007
	
	\bibitem{Ivanyshyn2} O. Ivanyshyn and R. Kress, Identification of sound-soft 3D obstacles from phaseless data  {\it Inverse Probl. Imaging}  {\bf 4}, 131--149, 2010.
	
	\bibitem{Ivanyshyn3} O. Ivanyshyn and R. Kress, Inverse scattering for surface impedance from phaseless far field data, {\it J. Comput. Phys.}  {\bf 230}, 3443--3452, 2011.
	
	\bibitem{KarageorghisAPNUM} A. Karageorghis, B.T. Johansson and  D. Lesnic, The method of fundamental solutions for the identification of a sound-soft obstacle in inverse acoustic scattering {\it Applied Numerical Mathematics} {\bf 62}, 1767--1780, 2012.
	
	\bibitem{Kli14} M. V. Klibanov, Phaseless inverse scattering problems in three dimensions {\it SIAM J. Appl. Math.}, {\bf 74}, 392--410, 2014.
	
	\bibitem{Kli17} M. V. Klibanov, A phaseless inverse scattering problem for the 3-D Helmholtz equation {\it Inverse Probl. Imaging}, {\bf 11}, 263--276, 2017.
	
	\bibitem{KR16} M. V. Klibanov and V. G. Romanov, Reconstruction procedures for two inverse scattering problems without the phase information {\it SIAM J. Appl. Math.}, {\bf 76}, 178--196, 2016.
	
	\bibitem{KR17} M. V. Klibanov and V. G. Romanov, Uniqueness of a 3-D coefficient inverse scattering problem without the phase information {\it Inverse Problems}, {\bf 33}, 095007, 2017.
	
	\bibitem{Kress} R. Kress and W. Rundell, Inverse obstacle scattering with modulus of the far field pattern as data {\it Inverse Problems in Medical Imaging and Nondestructive Testing (Oberwolfach, 1996)}, 75--92, 1997.
	
	\bibitem{Lee2016}  K. M. Lee, Shape reconstructions from phaseless data, {\it Eng. Anal. Bound. Elem.}  {\bf 71}, 174--178, 2016.
	
	\bibitem{Li1} J. Li and H. Liu, Recovering a polyhedral obstacle by a few backscattering measurements {\it J. Differential Equat.}  {\bf 259}, 2101--2120, 2015.
	
	\bibitem{LLW17} J. Li, H. Liu and Y. Wang, Recovering an electromagnetic obstacle by a few phaseless backscattering measurements {\it Inverse Problems}, {\bf 33}, 035001, 2017.
	
	\bibitem{Li} J. Li, H. Liu and J. Zou, Strengthened linear sampling method with a reference ball {\it SIAM J. Sci. Comput.}, {\bf 31}(6) 4013--4040, 2009.
	
	\bibitem{Liu} J. Liu and J. Seo, On stability for a translated obstacle with impedance boundary condition {\it Nonlinear Anal.} {\bf 59}, 731--744, 2004
	
	\bibitem{LZ09} X. Liu and B. Zhang, Unique determination of a sound soft ball by the modulus of a single far field datum {\it J. Math. Anal. Appl.}, {\bf 365}, 619--624, 2009.
	
	\bibitem{McLean} W. McLean, {\it Strongly Elliptic Systems and Boundary Integral Equations} (Cambridge: Cambridge University), 2000.
	
	\bibitem{Pot01} R. Potthast, {\it Point Sources and Multipoles in Inverse Scattering Theory} (London: Chapman \& Hall), 2001.
	
	\bibitem{Shin} J. Shin, Inverse obstacle backscattering problems with phaseless data {\it Euro. J. Appl. Math.} {\bf 27}, 111--130, 2016.
	
	\bibitem{ZhangBo20171} B. Zhang and H. Zhang, Recovering scattering obstacles by multi-frequency phaseless far-field data, {\it Journal of Computational Physics}, {\bf 345}, 58--73, 2017.
	
	\bibitem{ZhangBo20172} B. Zhang and H. Zhang, Imaging of locally rough surfaces from intensity-only far-field or near-field data {\it Inverse Problems} {\bf 33}, 055001, 2017.
	
	\bibitem{ZhangBo20173} B. Zhang and H. Zhang, Uniqueness in inverse scattering problems with phaseless far-field
	data at a fixed frequency {\it arXiv:1709.07878v1}, 2017.
	
\end{thebibliography}
\end{document}